\documentclass[12 pt]{article}
\usepackage{amsmath}
\usepackage{amsfonts}
\usepackage{amssymb}
\usepackage{amsthm}
\usepackage{hyperref}
\usepackage[style=alphabetic,sorting=nyt,maxnames=5,maxalphanames=4,backend=bibtex]{biblatex}

\newtheorem{thm}{Theorem}
\newtheorem{lem}{Lemma}
\newtheorem{conj}{Conjecture}

\title{A Note on the Orderability of Dehn Fillings of the Manifold $v2503$}
\author{Konstantinos Varvarezos}
\begin{document}
\maketitle

\begin{abstract}
We show that Dehn filling on the manifold $v2503$ results in a non-orderable space for all rational slopes in the interval $(-\infty , -1)$.  This is consistent with the L-space conjecture, which predicts that all fillings will result in a non-orderable space for this manifold.
\end{abstract}

\section{Introduction}
This paper studies the orderability of a certain 3-manifold in view of an outstanding conjectured relationship between orderability and L-spaces.

A \textit{left-ordering} on a group $G$ is a total ordering $\prec$ on the elements of $G$ that is invariant under left-multiplication; that is, $g \prec h$ implies $fg \prec fh$ for all $f,g,h \in G$.  A group is said to be \textit{left-orderable} if it is nontrivial and admits a left ordering.  A 3-manifold $M$ is called \textit{orderable} if $\pi_1(M)$ is left-orderable.

If $M$ is a rational homology 3-sphere, then the rank of its Heegaard Floer homology  is greater than or equal to the order of its first (integral) homology group.  $M$ is called an \textit{L-space} if equality holds; that is, if $\mathrm{rk}\big(\widehat{HF}(M)\big) = \left|H_1(M;\mathbf{Z})\right|$.

This work is motivated by the following proposed connection between L-spaces and orderability, first conjectured by Boyer, Gordon, and Watson.

\begin{conj}[\cite{BGW}] \label{conj:LS}
An irreducible rational homology 3-sphere is an L-space if and only if
its fundamental group is not left-orderable.
\end{conj}

In \cite{BGW}, this equivalence was shown to hold for all closed, connected, orientable,
geometric three-manifolds that are non-hyperbolic.

If $M$ is a rational homology solid torus, then a framing of the boundary $(\mu,\lambda)$  is called a \textit{homological framing} for $\partial M$ if $\lambda$ is (rationally) nullhomologous.  Given a framing on $\partial M$ and a reduced fraction $\frac{p}{q} \in \mathbf{Q} \cup \{\infty\}$, we denote the $\frac{p}{q}$ Dehn filling by $M \big(\frac{p}{q}\big)$.

Culler and Dunfield \cite{CD} have remarked that the cusped hyperbolic manifold $v2503$ has the property that every non-longitudinal Dehn filling is an L-space (the longitudinal filling is $S^1 \times S^2 \# \mathbf{RP}^3$).  Thus, if Conjecture \ref{conj:LS} holds, one would expect none of the Dehn fillings of $v2503$ to be orderable (the longitudinal filling is non-orderable as its fundamental group has torsion).  To that end, we prove the following partial result:

\begin{thm}\label{thm:neg}
Let $M = v2503$.  Then for a certain homological framing, $M(r)$ is not orderable for any rational slope $r \in (-\infty, -1)$.
\end{thm}

\subsection*{Acknowledgements}
The author would like to thank Professors Zolt\'{a}n Szab\'{o} and Peter Ozsv\'{a}th for suggesting this problem as well as for providing feedback on drafts of this paper.

\section{Ordering}
We note the following useful facts, which hold for any left-ordered group $(G,\prec)$:
\begin{itemize}
\item For each $g \in G$, $1 \prec g \Leftrightarrow g^{-1} \prec 1$
\item For all $a,b \in G$, $1 \prec a,b \Rightarrow 1 \prec ab$ and similarly $a,b \prec 1 \Rightarrow ab \prec 1$.
\end{itemize}

We also call any element $g$ of $G$ \textit{positive} whenever $1 \prec g$, and similarly, $g$ is said to be \textit{negative} if $g \prec 1$.

Let $M$ be a compact, connected, oriented irreducible 3-manifold with icompressible torus boundary, and let $(\mu, \lambda)$ be a framing for $\partial M$. In \cite{CW}, Clay and Watson describe a criterion for obstructing left-orderability of Dehn fillings of $M$. One corollary of that criterion is:
\begin{thm}[\cite{CW}]\label{thm:cw}
Let $\frac{p}{q},\frac{p_0}{q_0},\frac{p_1}{q_1}$ be rational numbers satisfying $\frac{p}{q} \in \big(\frac{p_0}{q_0},\frac{p_1}{q_1}\big)$ such that $q, q_0, q_1 > 0$ and $p, p_0, p_1 < 0$.  Suppose that $\pi_1(\partial M)$ is not sent to 1 by the quotient map $\pi_1(M) \rightarrow \pi_1\big(M\big(\frac{p}{q}\big)\big)$ and that for each left ordering $\prec$ of $\pi_1(M)$, $\mu^{p_0}\lambda^{q_0} \prec 1$ implies $\mu^{p_1}\lambda^{q_1} \prec 1$.  Then $\pi_1\big(M\big(\frac{p}{q}\big)\big)$ is not left-orderable.
\end{thm}

\begin{proof}[Remark]
This is essentially Corollary 2.2 in \cite{CW} except in that paper, $p, p_0, p_1$ are all required to be positive; however, their proof works just as well assuming they are all negative instead.  Alternatively, one can simply replace $\mu$ with $\mu^{-1}$ and apply their theorem directly, noting that the only necessary property of $\mu$ and $\lambda$ is that they generate $\pi_1(\partial M)$.
\end{proof}

\section{The Manifold $v2503$}
Now let us turn our attention to the manifold named $v2503$ in the SnapPy census \cite{SnapPy}, which we denote $M$ for the rest of this section.  It is also known as $M7_{2459}$ in the nomenclature of \cite{CHW}.  $M$ is a hyperbolic 3-manifold with one toroidal cusp, and $M$ is also a rational homology solid torus.  Indeed, SnapPy gives that $H_1(M;\mathbf{Z})\cong \mathbf{Z}\oplus \mathbf{Z}/10\mathbf{Z}$.
\subsection*{Fundamental Group}
According to SnapPy, the fundamental group of $M = v2503$ has the following presentation:
\begin{equation}\label{pi1}
\pi_1(M) = \left\langle a,b \vert a^2 b^{-2} a b^{-2} a^2 b a^2 b a b a^2 b = 1 \right\rangle
\end{equation}
In addition, SnapPy also gives that the ``meridian" $m$ and ``longitude" $l$ are:
\begin{align*}
m &= b^{-1} a^2 b a^2 \\
l &= b^{-2} a b^{-2} a b^{-1}
\end{align*}
We follow the convention of Culler and Dunfield \cite{CD} for the   homological framing.  In particular, our homological meridan $\mu$ and homological hongitude $\lambda$ correspond to $(0,1)$ and $(-1,0)$ respectively in SnapPy's framing . That is:
\begin{align}
\mu &= l = b^{-2} a b^{-2} a b^{-1} \label{mer} \\
\lambda &= m^{-1} = a^{-2} b^{-1} a^{-2} b \label{lon}
\end{align}

Notice that, by considering the abelianisation of (\ref{pi1}), the generator $a$ corresponds to a generator of the torsion subgroup of $H_1(M;\mathbf{Z})$, whereas $b$ is a free generator.  Moreover, $[\mu] = [a]^2[b]^{-5} \in H_1(M;\mathbf{Z})$ and $[\lambda] = [a]^{-4}$, and so $\lambda$ is rationally nullhomologous, which is consistent with its being a homological longitude.


For convenience, let us put:
\begin{align}
x &= b^{-2} a \label{x} \\
y &= b a^2 \label{y}
\end{align}
We record for later the following:
\begin{align}
a^2 x^2 a b a^2 b a b a^2 b &= 1 \label{pos}\\
\mu &= x^2 b^{-1} \label{mer2}\\
\lambda &= y^{-2} b^2 \label{lon2}\\
\lambda &= b a b a^2 b a^2 b^{-2} a b^{-1} = b a y^2 x b^{-1} \label{altl} \\
\mu^{-1} \lambda = \lambda \mu^{-1} &= b a y^2 x^{-1} \label{m-1l}\\
 &= y a^{-1} y^2 x^{-1} \tag{\ensuremath{10'}} \label{m-1l2} 
\end{align}

Apart from (\ref{altl}), these are straightforward consequences of (\ref{pi1})--(\ref{y}).  To see why (\ref{altl}) holds, observe that the group relation in (\ref{pi1}) can be rewritten as: 
\begin{align*}
1 &= a^2 b^{-2} a b^{-1} \left(b^{-1} a^2 b a^2 \right) b a b a^2 b \\
&= a^2 b^{-2} a b^{-1} \left(a^{-2} b^{-1} a^{-2} b\right)^{-1} b a b a^2 b\\
&=  a^2 b^{-2} a b^{-1} \lambda^{-1} b a b a^2 b
\end{align*}
where (\ref{lon}) was used in the last step to substitute for $\lambda$.  Now the desired expression follows by isolating $\lambda$ in the equation above.

\subsection*{Orderability constraints for $v2503$}
We now use the information about the fundamental group of $v2503$ to prove the following observations, which are the basic ingredients for the proof of the main theorem.
\begin{lem}\label{lem:powers}
Let $\prec$ be a left ordering of $\pi_1(v2503)$.  If $\mu^{-1} \lambda \prec 1$ then $\mu^{-n} \lambda \prec 1$ for all $n \geq 1$.
\end{lem}
\begin{proof}
Suppose that $\mu^{-1} \lambda \prec 1$. There are four cases, depending on the signs of the generators $a$ and $b$.

Case I: $b \prec 1 \prec a$.  In this case, $1 \prec \mu$ since, by (\ref{mer}), $\mu$ can be expressed as the product of positive terms.  Hence, $\mu^{-1} \prec 1$ and so for each $n\geq1$, $\mu^{-n} \lambda \prec 1$ as it is  the product of negative terms.

Case II: $a,b \prec 1$.  Notice that, by (\ref{pos}), it must hold that $1 \prec x$ for otherwise, 1 would be expressed as the product of negative terms.  Now by (\ref{mer2}), we see that $mu$ is the product of positive terms, and hence $1 \prec \mu$.  As in Case I, we once again have $\mu^{-n} \lambda \prec 1$ for all $n\geq1$.

Case III: $1 \prec a,b$.  In this case, we see from (\ref{y}) that $1 \prec y$ as $y$ is the product of positive terms.  On the other hand, we have that $x \prec 1$ for otherwise, 1 would be expressed as the product of positive terms in (\ref{pos}).  But then, by (\ref{m-1l}), we see that $\mu^{-1} \lambda$ is expressed as a product of positive terms, contradicting the hypothesis that $\mu^{-1} \lambda \prec 1$.  So this case cannot happen.

Case IV: $a \prec 1 \prec b$.  In (\ref{x}), we see $x$ expressed as the product of negative terms, and so $x \prec 1$.  Now, by (\ref{m-1l2}), we conclude that $y \prec 1$ as otherwise, $\mu^{-1} \lambda$ would be the product of positive terms, contradicting the hypothesis that $\mu^{-1} \lambda \prec 1$. Now, by (\ref{lon2}), $1 \prec \lambda$ because $\lambda$ is expressed as the product of positive terms.  The hypothesis that $\mu^{-1} \lambda \prec 1$ implies, by invariance under left-multiplication, that $\lambda \prec \mu$.  Hence, $1 \prec \mu$, but, from (\ref{mer}) we see $\mu$ as a product of positive elements, a contradiction.  So this case, too, cannot happen.
\end{proof}

\begin{lem}\label{lem:nontrivial}
Let $r \in \mathbf{Q}$.   If  $\pi_1(\partial M)$ is sent to 1 by the quotient map $\pi_1(M) \rightarrow \pi_1(M(r))$, then $M(r) $ is not orderable.
\end{lem}
\begin{proof}
If the subgroup $\pi_1(\partial M)$ of $\pi_1(M)$ is sent to 1 by the quotient map, then that map factors as: $\pi_1(M) \rightarrow \left\langle \pi_1(M) \vert \mu = 1, \lambda = 1 \right\rangle \rightarrow \pi_1(M(r))$.  Let us examine the group $G = \left\langle \pi_1(M) \vert \mu = 1, \lambda = 1 \right\rangle$. By (\ref{mer2}), (\ref{lon2}), and (\ref{m-1l}), we see that the following relations hold in $G$:
\begin{align}
b &= x^2 \label{relbx} \\
b^2 &= y^2 \label{relby} \\
x &= bay^2 \label{relbayx}
\end{align}
Notice further that (\ref{pos}) can be re-written as
\begin{equation}
x^2ay^2a^{-1}bay^2=1 \nonumber
\end{equation}
This becomes, using (\ref{relbx}) and (\ref{relbayx}):
\begin{equation}
xa^{-1}y^2=1 \nonumber
\end{equation}
Using (\ref{relbx}) and (\ref{relby}), this becomes:
\begin{align*}
xa^{-1}x^4 = 1 \\
a = x^5
\end{align*}
Hence, recalling (\ref{x}) and (\ref{y}), $G$ has the following presentation:
\begin{equation}\nonumber
G = \left\langle a,b,x,y \vert a=x^5, b=x^2, x=b^{-2}a, y=ba^2, b^2=y^2, x=bay^2 \right\rangle
\end{equation}
This can be simplified to:
\begin{align*}
G &= \left\langle x,y \vert y=x^{12}, x^4=y^2, x=x^7y^2 \right\rangle \\
& = \left\langle x \vert  1=x^{20}, 1=x^{30} \right\rangle \\
& = \left\langle x \vert  1=x^{10} \right\rangle \\
& \cong \mathbf{Z}/10\mathbf{Z}
\end{align*}
Therefore, as $\pi_1(M(r))$ is the quotient of a finite group, it is finite as well, and hence not left-orderable (recall that, by convention, the trivial group is considered not left-orderable).
\end{proof}

We are now ready to prove our main result.

\begin{proof}[Proof of Theorem \ref{thm:neg}]
Let $r \in \mathbf{Q} \cap (-\infty, -1)$. By Lemma \ref{lem:nontrivial}, we may assume that $\pi_1(\partial M)$ is not sent to 1 by the quotient map $\pi_1(M) \rightarrow \pi_1(M(r))$.  Furthermore, as $M$ is hyperbolic, it is irreducible and has incompressible torus boundary. Then, since $r \in (-n, -1)$ for some integer $n \geq 1$, Lemma \ref{lem:powers} together with Theorem \ref{thm:cw} tells us that $M(r)$ is not orderable, as required.
\end{proof}

\printbibliography

\end{document}